\documentclass[12pt,reqno]{amsart}

\usepackage{amsmath}
\usepackage{amssymb}
\usepackage{amsthm}
\usepackage{amscd}
\usepackage{xypic}

\headsep=1cm \numberwithin{equation}{section}
\def\G{{\mathbb G}}
\def\K{{\mathbb K}}

\def\P{{\mathbb P}}

\def\Z{{\mathbb Z}}
\newtheorem{theorem}{Theorem}[section]
\newtheorem{lemma}[theorem]{Lemma}

\newtheorem{proposition}[theorem]{Proposition}
\newtheorem{corollary}[theorem]{Corollary}

\theoremstyle{definition}

\newtheorem{notation and remakrs}[theorem]{Notation and Remarks}
\newtheorem{convention and reminder}[theorem]{Convention and Reminder}
\newtheorem{convention and remark}[theorem]{Convention and Remark}
\newtheorem{definition and remark}[theorem]{Definition and Remark}

\newtheorem{reminders and definition}[theorem]{Reminders and Definition}

\newtheorem{remark and notations}[theorem]{Remark and Notations}
\newtheorem{notation and remark}[theorem]{Notation and Remark}
\newtheorem{example}[theorem]{Example}

\newcommand\Ker{\operatorname{\Ker}}

\begin{document}

\title{On rank $3$ quadratic equations of Veronese embeddings}

\author[E. Park]{Euisung Park}
\email{euisungpark@korea.ac.kr}
\address{Euisung Park, Department of Mathematics, Korea University, Seoul 136-701, Republic of Korea}

\author[S. Sim]{Saerom Sim}
\email{saeromiii@korea.ac.kr}
\address{Saerom Sim, Department of Mathematics, Korea University, Seoul 136-701, Republic of Korea}


\maketitle

\begin{abstract}
This paper studies the geometric structure of the locus $\Phi_3 (X)$ of rank $3$ quadratic equations of the Veronese variety $X = \nu_d (\P^n)$. Specifically, we investigate the minimal irreducible decomposition of $\Phi_3 (X)$ of rank $3$ quadratic equations and analyze the geometric properties of the irreducible components of $\Phi_3 (X)$ such as their desingularizations. Additionally, we explore the non-singularity and singularity of these irreducible components of $\Phi_3 (X)$.
\end{abstract}

\section{Introduction}
\noindent Throughout this paper, we work over an algebraically closed field $\K$ of characteristic $0$. Let
$$X = \nu_d (\P^n ) \subset \P^r ,\quad r = {{n+d} \choose {n}}-1 ,$$
be the $d$th Veronese embedding of $\P^n$, and let $I(X)$ be the homogeneous ideal of $X$ in the homogeneous coordinate ring $S = \K [z_0 , z_1 , \ldots , z_r ]$ of $\P^r$. Various structural properties of $I(X)$ have been studied so far. For example, $I(X)$ is generated by quadrics, and the first few syzygy modules are generated by linear syzygies (cf. \cite{BCR}, \cite{Gr1}, \cite{Gr2}, \cite{JPW}, \cite{OP}, etc). Also, $X$ is \textit{determinantally presented}, meaning that $I(X)$ is generated by $2$-minors of a matrix of linear forms in $S$ (cf. \cite{EKS}, \cite{Pu}, \cite{Su}, \cite{SS}, etc). This implies that $I(X)$ is generated by quadrics of rank $\leq 4$. The authors in \cite{HLMP} improved this result by showing that $I(X)$ can always be  generated by quadrics of rank $3$. Now, let
\begin{equation*}
\Phi_3 (X) := \{ [Q] ~|~ Q \in I(X)_2 - \{0\},~ {\rm rank}(Q) = 3 \} \subset \mathbb{P} (I(X)_2 ).
\end{equation*}
be the locus of rank $3$ quadratic equations of $X$. Recently, some basic geometric structures of the projective algebraic set $\Phi_3 (X)$ were shown in \cite{Park}. To be precise, there are finite morphisms
\begin{equation}\label{eq:normalization}
\widetilde{Q_{\ell}} : \mathbb{G} (1,\P^p ) \times \P^q \rightarrow \Phi_3 (X)
\end{equation}
for $1 \le \ell \le d/2$, where $p =  h^0 (\P^n , \mathcal{O}_{\P^n} (\ell)) -1$ and $q = h^0 (\P^n , \mathcal{O}_{\P^n} (d-2\ell)) -1$. Also, letting $W_{\ell} (X)$ be the image of $\mathbb{G} (1,\P^p ) \times \P^q$ by $\widetilde{Q_{\ell}}$, it holds that
\begin{equation}\label{eq:natural irred decomp}
\Phi_3 (X) = \bigcup_{1 \le \ell \le d/2} W_{\ell} (X)
\end{equation}
is an irreducible decomposition of $\Phi_3 (X)$ (cf. \cite[Theorem 1.2]{Park}).

Along this line, the main purpose of this paper is to study the following four basic problems about the irreducible components $W_{\ell} (X)$ of $\Phi_3 (X)$ for $1 \le \ell \le d/2$.\\

\begin{enumerate}
\item[$(i)$] Prove that (\ref{eq:natural irred decomp}) is the minimal irreducible decomposition of $\Phi_3 (X)$.
\item[$(ii)$] Prove that the finite morphism
\begin{equation}\label{eq:Q-map normalization}
\widetilde{Q_{\ell}} : \mathbb{G} (1,\P^p ) \times \P^q \rightarrow W_{\ell} (X)
\end{equation}
induced from (\ref{eq:normalization}) is generically injective and hence the normalization of $W_{\ell} (X)$.
\item[$(iii)$] Prove that $W_{\ell} (X)$ is nondegenerate in $\mathbb{P} (I(X)_2 )$, or equivalently, that $I(X)$ can be generated by quadratic equations in $W_{\ell} (X)$.
\item[$(iv)$] Determine whether $W_{\ell} (X)$ is smooth or singular. \\
\end{enumerate}

First, let us summarize the known facts about these four problems. For $n \geq 2$, problems $(i)$ and $(ii)$ are solved in \cite[Theorem 1.6 and Corollary 1.7]{Park} and \cite[Corollary 3.4]{Park}, respectively. The reason they are only proven for $n \geq 2$ is that their proofs rely on the fact that a general hypersurface of degree $\ell$ of $\P^n$ is irreducible. Problem $(iii)$ is solved only for the case $\ell = 1$. For details, we refer the reader to \cite[Theorem 4.2]{HLMP}. Regarding problem $(iv)$, it is shown in \cite[Theorem 3.7]{Park} that $\Phi_3 (\nu_2 (\P^n ) )= W_1 (\nu_2 (\P^n))$ is isomorphic to the second Veronese embedding of the Plücker embedding of $\mathbb {G} (1,\P^n)$. However, problem $(iv)$ has not yet been studied for $d \geq 3$. \\

Our first main result solves the problems $(i)$ and $(iii)$ when $X$ is a rational normal curve.

\begin{theorem}\label{thm:minimal and birational}
Let $\mathcal{C} \subset \P^d$ be a rational normal curve of degree $d$. Then

\renewcommand{\descriptionlabel}[1]%
             {\hspace{\labelsep}\textrm{#1}}
\begin{description}
\setlength{\labelwidth}{13mm} \setlength{\labelsep}{1.5mm}
\setlength{\itemindent}{0mm}

\item[{\rm (1)}] The minimal irreducible decomposition of $\Phi_3 (\mathcal{C})$ is
$$\Phi_3 (\mathcal{C}) = \bigcup_{1 \le \ell \le d/2} W_{\ell} (\mathcal{C}).$$

\item[{\rm (2)}] For every $1 \le \ell \le d/2$, the projective variety $W_{\ell} (\mathcal{C})$ is non-degenerate in $\P (I(\mathcal{C})_2 )$. Therefore, $I(\mathcal{C})$ can be generated by quadrics in $W_{\ell} (\mathcal{C})$.

\item[{\rm (3)}] $W_1 (\mathcal{C})$ is projectively equivalent to $\nu_2 \left( \P^{d-2} \right)$. In particular, $W_1 (\mathcal{C})$ is smooth.
\end{description}
\end{theorem}

The proof of Theorem \ref{thm:minimal and birational} is in Section 4.

As mentioned above, problem $(i)$ has already been solved for $n \geq 2$, but its proof does not apply to the case $n=1$. Also, problem $(iii)$ has been proven only for $\ell =1$. Our new approach in this paper for the case $n=1$ is to investigate the coefficients of the map $\widetilde{Q_{\ell}}$ with respect to the standard quadratic generators of $I(\mathcal{C})$. These coefficients are polynomials in several variables, but expressing them explicitly is too complex to be practical. Instead, we were able to solve the problems $(i)$ and $(iii)$ by verifying some useful properties of these polynomials.\\

Next, we consider problems $(ii)$ and $(iv)$.

\begin{theorem} \label{thm:Singularity of W_ell (X))}
Let $X=\nu_d (\P^n)$ be as above. Suppose that $d=2e+1$ or $d=2e$. Then

\renewcommand{\descriptionlabel}[1]%
             {\hspace{\labelsep}\textrm{#1}}
\begin{description}
\setlength{\labelwidth}{13mm} \setlength{\labelsep}{1.5mm}
\setlength{\itemindent}{0mm}

\item[{\rm (1)}] The morphisms $\widetilde{Q_{\ell}}$ in (\ref{eq:Q-map normalization}) is the normalization (and so the desingularization) morphism.
\item[{\rm (2)}] The morphisms $\widetilde{Q_1}$ and $\widetilde{Q_e}$ are injective.
\item[{\rm (3)}] If $2 \leq \ell \leq e-1$, then $W_{\ell} (X)$ is singular and
$$\dim \ {\rm Sing} (W_\ell (X)) \geq \theta (n,d,m)$$ where $\theta (n,d,m)$ is defined to be the following integer:
\begin{align*}
\max_{1 \leq m \leq min\{\ell, \lfloor \frac{d-2\ell}{2} \rfloor \}} \left\{2\binom{n+m}{n}+2\binom{n+\ell-m}{n}+\binom{n+d-2\ell-2m}{n}-7 \right\}.
\end{align*}
In particular, if $n=1$ and hence $\dim \ W_\ell (X) = d-2$ then
$$\dim \ {\rm Sing} (W_\ell (X)) \geq d-4.$$
\end{description}
\end{theorem}

The proof of Theorem \ref{thm:Singularity of W_ell (X))} is in Section $5$.

As mentioned earlier, problem $(ii)$ has been solved for $n \geq 2$, but its proof does not apply to the case $n=1$. Along this line, we will provide a proof of problem $(ii)$ that applies for all $n \geq 1$. Since problem $(ii)$ is completely proven in Theorem \ref{thm:Singularity of W_ell (X))}.$(1)$, we could almost solve problem $(iv)$ by investigating the locus where the birational morphism $\widetilde{Q_{\ell}}$ fails to be one-to-one.\\

\noindent {\bf Organization of the paper.}
In Section $2$ we set up notation for the Grassmannian varieties and discuss their quadratic generators. Section $3$ is devoted to establishing some useful facts about the $\widetilde{Q_\ell}$-maps for the rational normal curve. In Section $4$ and $5$, we prove Theorem \ref{thm:minimal and birational} and Theorem \ref{thm:Singularity of W_ell (X))}, our main theorems regarding the geometric structure of the locus $\Phi_3 (X)$ of rank $3$ quadratic equations of the Veronese variety $X = \nu_d (\P^n)$. Finally, in Section $6$, we study the problem of determining whether the map $\widetilde{Q_e}$ is an isomorphism in a few examples.  \\

\noindent {\bf Acknowledgement.} This work was supported by the National Research Foundation of Korea(NRF) grant funded by the Korea government(MSIT) (No. 2022R1A2C1002784).

\section{Preliminaries}
\noindent We denote by $\mathbb{G} (1,\P^n )$ the Grassmannian manifold of lines in $\P^n$. Also let
$$Z \subset \mathbb{P}^N , N=\binom{n+1}{2}-1$$
be the Plücker embedding of $\mathbb{G} (1,\P^n )$. In this section we recall the standard expression of the homogeneous ideal and the homogeneous coordinate ring of $Z$, which will be used to understand the finite morphisms $\widetilde{Q_{\ell}} : \mathbb{G} (1,\P^p ) \times \P^q \rightarrow W_{\ell} (X)$ in (\ref{eq:normalization}). We refer the reader to \cite{Mu} for details.

\begin{notation and remark}\label{not and rmk:section2}
Let $R = \K [A_0 , A_1 , \ldots , A_n , B_0 , B_1 , \ldots , B_n ]$ be the polynomial ring in $2n+2$ variables over $\K$, and let
\[
M=
\begin{pmatrix}
A_0 & A_1 & \cdots & A_\alpha & \cdots & A_\beta & \cdots & A_n \\
B_0 & B_1 & \cdots & B_\alpha & \cdots & B_\beta & \cdots & B_n
\end{pmatrix} .
\]

\renewcommand{\descriptionlabel}[1]%
             {\hspace{\labelsep}\textrm{#1}}
\begin{description}
\setlength{\labelwidth}{13mm} \setlength{\labelsep}{1.5mm}
\setlength{\itemindent}{0mm}

\item[{\rm (1)}] The homogeneous coordinate ring of $Z$ is isomorphic to the subring
\begin{align*}
T = \mathbb{K}\left[\left\{q_{\alpha \beta} \mid 0 \leq \alpha < \beta \leq n \right\}\right]
\end{align*}
of $R$ where $q_{\alpha \beta}$ denotes the $2$-minor $A_{\alpha} B_{\beta} - A_{\beta} B_{\alpha}$ of $M$.

\item[{\rm (2)}] The $\K$-vector space $T_2$ is spanned by the subset $\mathcal{Q} := \mathcal{Q}_2 \cup \mathcal{Q}_3 \cup \mathcal{Q}_4$ where
\begin{equation}
\mathcal{Q}_s := \left\{q_{\alpha \beta} q_{\gamma \delta} \mid q_{\alpha \beta}, q_{\gamma \delta} \in T_1 \ \mbox{and} \  \#\{\alpha, \beta, \gamma, \delta\} =s  \right\}
\end{equation}
for $s=2,3,4$. Obviously,
$$\mathcal{Q}_2 = \{ q_{\alpha \beta} ^2 \mid 0 \le \alpha < \beta \le n \},$$
$$\mathcal{Q}_3 = \bigcup_{0 \leq \alpha < \beta < \gamma  \le n} \{ q_{\alpha \beta} q_{\alpha \gamma} ,  q_{\alpha \beta} q_{\beta \gamma} , q_{\alpha \gamma} q_{\beta \gamma}   \}$$
and
$$\mathcal{Q}_4 = \bigcup_{0 \leq \alpha < \beta < \gamma < \delta \le n} \{ q_{\alpha \beta} q_{\gamma \delta} , q_{\alpha \gamma} q_{\beta \delta} ,  q_{\alpha \delta} q_{\beta \gamma}  \}.$$
In addition, $\mathcal{Q}$ contains a basis for $T_2$.
\end{description}
\end{notation and remark}

\begin{theorem}[Theorem 8.15 in \cite{Mu}]\label{thm:Grass Plucker relations}
Let $Z \subset \P^N$ be as above, and let
$$\{ p_{\alpha \beta} ~\mid~ 0 \leq \alpha < \beta \leq n \}$$
be the set of homogeneous coordinates of $\mathbb{P}^N$. Then the homogeneous ideal of $Z$ is minimally generated by the set
\begin{align} \label{plurel}
\{ p_{\alpha \beta} p_{\gamma \delta} - p_{\alpha \gamma} p_{\beta \delta} + p_{\alpha \delta} p_{\beta \gamma} ~|~0 \leq \alpha<\beta<\gamma<\delta \leq n \}
\end{align}
of the ${n+1} \choose {4}$ Pl$\ddot{u}$cker relations.
\end{theorem}

Now, we can obtain a basis of the $\K$-vector space $T_2$ by combining Notations and Remark \ref{not and rmk:section2} and Theorem \ref{thm:Grass Plucker relations}.

\begin{proposition} \label{basisplu}
The set $\mathcal{B} := \left\{q_{\alpha \beta} q_{\gamma \delta} \mid q_{\alpha \beta} , q_{\gamma \delta}  \in T_1 \ \mbox{and} \ \alpha \leq \gamma \leq \beta, \delta  \right\}$ is a basis of $T_2$.
\end{proposition}

\begin{proof}
Obviously, $\mathcal{B}$ is a subset of $\mathcal{Q}$ which contains $\mathcal{Q}_2$ and $\mathcal{Q}_3$. Moreover,
$$\mathcal{B} \cap \mathcal{Q}_4 = \bigcup_{0 \leq \alpha < \beta < \gamma < \delta \le n} \{  q_{\alpha \gamma} q_{\beta \delta} ,  q_{\alpha \delta} q_{\beta \gamma}  \}.$$
Also, if $0 \leq \alpha < \beta < \gamma < \delta \le n$ then the element $q_{\alpha \beta} q_{\gamma \delta}$ of $\mathcal{Q}$ can be written as
$$q_{\alpha \beta} q_{\gamma \delta} = q_{\alpha \gamma} q_{\beta \delta} - q_{\alpha \delta} q_{\beta \gamma}$$
by Theorem \ref{thm:Grass Plucker relations}. In consequence, it follows that $\mathcal{B}$ generates $T_2$ since so does $\mathcal{Q}$. On the other hand, the dimension of the $\K$-vector space $T_2$ is given as follows since $Z$ is arithmetically Cohen-Macaulay:
\begin{align*}
\dim_{\K}~ T_2 &= h^0 (\P^N , \mathcal{O}_{\P^N} (2)) - h^0 (Z , \mathcal{O}_{Z} (2)) \\
               &= {{N+1} \choose 2} - \binom{n+1}{4}  =  \binom{\binom{n+1}{2}+1}{2} - \binom{n+1}{4} = \frac{1}{3} \binom{n+2}{2} \binom{n+1}{2}
\end{align*}
Also, we have
\begin{align*}
|\mathcal{B}| & \leq |\mathcal{B}_2| +  |\mathcal{B}_3| +  |\mathcal{B} \cap \mathcal{Q}_4 | \\
              &\leq {{n+1} \choose {2}} + 3 \times {{n+1} \choose {3}} + 2 \times {{n+1} \choose {4}} =  \frac{1}{3} \binom{n+2}{2}  \binom{n+1}{2}.
\end{align*}
This completes the proof that $\mathcal{B}$ is a basis of $T_2$.
\end{proof}

\section{The $\widetilde{Q_{\ell}}$-maps of rational normal curves}
\noindent Throughout this section, let $\mathcal{C} \subset \mathbb{P}^d$, $d \geq 2$, be the standard rational normal curve and let $S = \K [z_0 , z_1 , \ldots , z_d ]$ denote the homogeneous coordinate ring of $\P^d$. Thus
$$\mathcal{C} = \{ [x^d:\cdots:x^{d-i} y^i :\cdots:y^d] ~|~ [x:y] \in \mathbb{P}^1 \}$$
and the homogeneous ideal $I(\mathcal{C})$ of $\mathcal{C}$ is minimally generated by the set
\begin{align} \label{quadbasis}
\{  Q_{i, j} := z_{i} z_{j+1} - z_{i+1} z_{j} \mid 0 \leq i < j \leq d-1  \}
\end{align}
of ${d \choose 2}$ quadratic polynomials.
The main purpose of this section is to make some preparations for the proof of Theorem \ref{thm:minimal and birational}.
We first establish a few notations that we will use throughout the rest of this paper.

\begin{definition and remark}\label{definition of Q map}
$(1)$ Let $\varphi : \K [x,y]_d \rightarrow S_1$ be the isomorphism of $\K$-vector spaces such that $\varphi (x^{d-i} y^i ) = z_i$ for all $0 \leq i \leq d$.
For each integer $1 \leq \ell \leq d/2$, we define the map
$$Q_{\ell} : \K [x,y]_{\ell} \times \K [x,y]_{\ell} \times \K [x,y]_{d-2\ell} \rightarrow I(\mathcal{C})_2$$
by
\begin{equation*}
Q_{\ell} (f,g,h) = \varphi (f^2 h) \varphi (g^2 h) - \varphi (fgh)^2 .
\end{equation*}
Also we denote the image of $Q_{\ell}$ in $\P (I(\mathcal{C})_2 )$ by $W_{\ell} (\mathcal{C})$. That is,
\begin{equation*}
W_{\ell} (\mathcal{C}) := \{ [Q_{\ell} (f,g,h)]~|~(f,g,h) \in \K [x,y]_{\ell} \times \K [x,y]_{\ell} \times \K [x,y]_{d-2\ell} ,~Q_{\ell} (f,g,h) \neq 0 \}.
\end{equation*}
The map $Q_{\ell}$ is well-defined since the restriction of the quadratic polynomial $Q_{\ell} (f,g,h)$ to $\mathcal{C}$ becomes
\begin{equation*}
(f^2 h) \times (g^2 h) - (fg h)^2  = 0
\end{equation*}
(cf. \cite[Proposition 6.10]{E}). Thus $W_{\ell} (\mathcal{C})$ is contained in $\Phi_3 (\mathcal{C})$.

\noindent $(2)$ If we write the three polynomials $f,g \in \K [x,y]_{\ell}$ and $h \in \K [x,y]_{d-2\ell}$ as
\begin{align*}
f &=A_0 x^{\ell} + A_1 x^{\ell-1}y + \cdots + A_{\ell} y^{\ell} \\
g &=B_0 x^{\ell} + B_1 x^{\ell-1}y + \cdots + B_{\ell} y^{\ell} \\
h &=C_0 x^{d-2\ell} + C_1 x^{d-2\ell-1}y + \cdots + C_{d-2\ell}y^{d-2\ell}
\end{align*}
for some $A_i, B_j, C_k \in \mathbb{K}$, then $Q_{\ell} (f,g,h)$ can be expressed in two different ways as
\begin{equation} \label{alphapha}
Q_{\ell} (f,g,h) = \sum_{0 \leq i < j \leq d-1}  \alpha_{i,j} Q_{i,j}
\end{equation}
and
\begin{equation*}
Q_{\ell} (f,g,h) = \sum_{0 \leq s \leq t \leq d}  \beta_{s,t} z_s z_t
\end{equation*}
where $\alpha_{i,j}$'s and $\beta_{s,t}$'s are polynomials in $A_0 , \ldots , A_{\ell} , B_0 , \ldots , B_{\ell} , C_0 , \ldots , C_{d-2\ell}$. For convention, we define $\alpha_{i, j}=0$ if $i<0$ or $j \geq d$ or $i \geq j$.

\noindent $(3)$ By \cite[Theorem 1.2]{Park}, the map $Q_{\ell}$ induces a finite morphism
\begin{align*}
\widetilde{Q_{\ell}} : \mathbb{G}(1, \mathbb{P}^{\ell}) \times \mathbb{P}^{d-2\ell} \rightarrow \P \left( I(\mathcal{C} )_2 \right)
\end{align*}
such that
\begin{equation*}
\left( \widetilde{Q_{\ell}} \right)^* (\mathcal{O}_{\P \left( I(\mathcal{C} )_2 \right)} (1) ) = \mathcal{O}_{\mathbb{G} (1,\P^{\ell} ) \times \P^{d-2\ell} } (2,2).
\end{equation*}
In particular, $W_{\ell} (\mathcal{C})$ is a projective variety of dimension $d-2$. Also, all $\alpha_{i, j}$'s are elements of the $\mathbb{K}$-vector space
$$\Delta_{\ell} := H^0 (\mathbb{G}(1, \mathbb{P}^{\ell}) \times \mathbb{P}^{d-2\ell} , \mathcal{O}_{\mathbb{G} (1,\P^{\ell} ) \times \P^{d-2\ell} } (2,2)).$$
Moreover, $\{ \alpha_{i, j} \mid 0 \leq i < j \leq d-1 \}$ is a base point free linear series on the projective variety $\mathbb{G}(1, \mathbb{P}^{\ell}) \times \mathbb{P}^{d-2\ell}$, which defines the finite morphism
$$\widetilde{Q_{\ell}} : \mathbb{G}(1, \mathbb{P}^{\ell}) \times \mathbb{P}^{d-2\ell} \rightarrow \P \left( I(\mathcal{C} )_2 \right) = \P^{{d \choose 2}-1} .$$

\noindent $(4)$ By the Künneth formula, it holds that
$$\Delta_{\ell} \cong \mathbb{K} [ \{ q_{\alpha \beta} ~\mid~ 0 \leq \alpha < \beta \leq \ell \} ]_2 \otimes \mathbb{K}[C_0, \dots, C_{d-2\ell}]_2 .$$
For $0 \le \alpha , \beta , \gamma ,\delta \leq \ell$ and $0 \leq \lambda, \mu \leq d-2\ell$, we define $D(\alpha, \beta, \gamma, \delta, \lambda, \mu)$ as
$$D(\alpha, \beta, \gamma, \delta, \lambda, \mu) = C_{\lambda} C_{\mu} q_{\alpha \beta} q_{\gamma \delta} = C_{\lambda} C_{\mu} (A_{\alpha} B_{\beta} - A_{\beta} B_{\alpha} ) (A_{\gamma} B_{\delta} -A_{\delta} B_{\gamma} ).$$
Then $\Delta_{\ell}$ is spanned by the subset
$$\Sigma^{\star} := \{ D(\alpha, \beta, \gamma, \delta, \lambda, \mu) ~|~ 0 \leq \alpha,\beta,\gamma,\delta \leq \ell ,~0 \leq \lambda, \mu \leq d-2\ell \}.$$
Also, Proposition \ref{basisplu} shows that the subset
$$\Sigma := \{ D(\alpha, \beta, \gamma, \delta, \lambda, \mu) ~|~ 0 \leq \alpha \leq \gamma \leq \beta, \delta \leq \ell ,~\alpha \neq \beta,~\gamma \neq \delta,~0 \leq \lambda \leq \mu \leq d-2\ell \}$$
of $\Sigma^{\star}$ is a basis for $\Delta_{\ell}$.
\end{definition and remark}

From now on, we will study a few basic properties of the subset $\{ \alpha_{i, j} \mid 0 \leq i < j \leq d-1 \}$ of $\Delta_{\ell}$ which consists of the coefficients of the map $Q_{\ell}$.

\begin{lemma}\label{lem:basic coefficients}
Let $\alpha_{i, j}$'s and $\beta_{s,t}$'s be as in Definition and Remark \ref{definition of Q map}.$(2)$.
Then:

\renewcommand{\descriptionlabel}[1]%
             {\hspace{\labelsep}\textrm{#1}}
\begin{description}
\setlength{\labelwidth}{13mm}
\setlength{\labelsep}{1.5mm}
\setlength{\itemindent}{0mm}

\item[$(1)$] For any $0 \leq s < t \leq d-1$, it holds that
$$\beta_{s, t}=-\alpha_{s-1, t} + \alpha_{s, t-1} ;$$

\item[$(2)$] If $i+j \leq d-1$, then
$$\alpha_{i,j} = \sum_{k=0}^{i} \beta_{k, i+j+1-k} ;$$

\item[$(3)$] If $i+j \geq d$, then
$$\alpha_{i,j} = \sum_{k=j+1}^{d} \beta_{i+j+1-k,k}.$$
\end{description}
\end{lemma}

\begin{proof}
$(1)$ The equality comes immediately by comparing the following two expressions of $Q_{\ell} (f,g,h)$:
$$Q_{\ell} (f,g,h) = \sum_{0 \leq i < j \leq d-1} \alpha_{i, j} \thickspace Q_{i, j} = \sum_{0 \leq s \leq t \leq d} \beta_{s, t} \thickspace z_s z_t $$
Indeed, $z_s z_t$ appears only in $Q_{s-1, t} = z_{s-1} z_{t+1} - z_s z_t$ and $Q_{s, t-1} = z_s z_t - z_{s+1} z_{t-1}$.

\noindent $(2)$ The assertion comes immediately from the following equalities:
\begin{align} \label{betalpha1}
\begin{cases}
\beta_{0, i+j+1} &= \alpha_{0, i+j} \\
\beta_{1, i+j} &= -\alpha_{0, i+j} + \alpha_{1, i+j-1} \\
\beta_{2, i+j-1} &= -\alpha_{1, i+j-1} + \alpha_{2, i+j-2} \\
&\vdots \\
\beta_{i-1, j+2} &= -\alpha_{i-2, j+2} + \alpha_{i-1, j+1} \\
\beta_{i, j+1} &= -\alpha_{i-1, j+1} + \alpha_{i,j}
\end{cases}
\end{align}

\noindent $(3)$ The assertion comes immediately from the following equalities:
\begin{align} \label{betalpha2}
\begin{cases}
\beta_{i+j+1-d, d} &= \alpha_{i+j+1-d, d-1} \\
\beta_{i+j+2-d, d-1} &= -\alpha_{i+j+1-d, d-1} + \alpha_{i+j+2-d, d-2} \\
\beta_{i+j+3-d, d-2} &= -\alpha_{i+j+2-d, d-2} + \alpha_{i+j+3-d, d-3} \\
&\vdots \\
\beta_{i-1, j+2} &= -\alpha_{i-2, j+2} + \alpha_{i-1, j+1} \\
\beta_{i, j+1}   &= -\alpha_{i-1, j+1} + \alpha_{i, j}
\end{cases}
\end{align}
\end{proof}

\begin{proposition} \label{alphabetapro}
For any $s$ and $t$ with $0 \leq s < t \leq d-1$, let $\varSigma(s, t)$ and $\varLambda(s, t)$ be subsets of $\Delta_{\ell}$ defined as
$$\varSigma (s, t)  := \{ D(\alpha, \beta, \gamma, \delta, \lambda, \mu) \in \Sigma^{\star} \mid \lambda+\alpha+\gamma=s \thickspace \textup{and} \thickspace \mu+\beta+\delta=t  \}$$
and
$$ \varLambda (s, t) := \varSigma (s, t) \cap \Sigma  .$$
Also, let $V(s,t)$ be the subspace of $\Delta_{\ell}$ generated by $\varSigma(s, t)$. Then:

\renewcommand{\descriptionlabel}[1]%
             {\hspace{\labelsep}\textrm{#1}}
\begin{description}
\setlength{\labelwidth}{13mm}
\setlength{\labelsep}{1.5mm}
\setlength{\itemindent}{0mm}

\item[$(1)$] $\beta_{s, t}$ is contained in $V(s, t)$;

\item[$(2)$] $\varLambda(s, t)$ is $\mathbb{K}$-linearly independent;

\item[$(3)$] Let $\varGamma$ be a basis of $V(s,t)$ that contains $\varLambda(s, t)$ and is contained in $\varSigma (s, t)$. Then, when $\beta_{s, t}$ is expressed as a $\mathbb{K}$-linear combination of the elements in $\varGamma$, all the coefficients of the elements in $\varLambda (s, t)$ are nonzero.
\end{description}
\end{proposition}

\begin{proof}
\noindent $(1)$  By using Definition and Remark \ref{definition of Q map}.$(3)$, $(4)$ and Lemma \ref{lem:basic coefficients}.$(1)$,
we know that $\beta_{s, t}$ can be expressed as a $\mathbb{K}$-linear combination of $D(\alpha, \beta, \gamma, \delta, \lambda, \mu) \in \Sigma^{\star}$, as shown below:
\[
\beta_{s, t} =  \sum d(\alpha, \beta, \gamma, \delta, \lambda, \mu) D(\alpha, \beta, \gamma, \delta, \lambda, \mu)
\]
where $d(\alpha, \beta, \gamma, \delta, \lambda, \mu) \in \mathbb{K}$.
Thus the remaining task is to show that the indices (i.e., $\alpha, \beta, \gamma, \delta, \lambda, \mu)$ in this linear combination satisfy the following condition:
\begin{align} \label{sumcon}
    \lambda+\alpha+\gamma=s ; \quad \mu+\beta+\delta=t .
\end{align}
The strategy of the proof is to show that, given $D(\alpha, \beta, \gamma, \delta, \lambda, \mu)$, it can be replaced by another $D:=D(-, -, -, -, -, -)$ that satisfies the condition \eqref{sumcon}.
This will be proven in two steps, as outlined below.

The first step is to determine the equation that the indices of the given $D(\alpha, \beta, \gamma, \delta, \lambda, \mu)$ must satisfy.
To do this, let us consider one of four terms that appear when expanding $D(\alpha, \beta, \gamma, \delta, \lambda, \mu)$, denoted as $D_1$ below:
\[
D_1:=C_{\lambda} C_{\mu} A_\alpha  A_\gamma B_\beta  B_\delta .
\]
By the construction of $Q_\ell$-mapping, the term $D_1$ can be obtained from $\varphi(f^2h)\varphi(g^2h)$ or $\varphi(fgh)^2$.
Thus, we will examine two cases below. But first, for convenience, let us divide the six indices of $D_1$ as below:
\[
\mathcal{I}_A := \{ \alpha, \gamma \}, \quad \mathcal{I}_B := \{ \beta, \delta \}, \quad \mathcal{I}_C := \{ \lambda, \mu \} .
\]
\noindent \underline{Case 1:} Suppose that $D_1$ is obtained from $\varphi(f^2h)\varphi(g^2h)$.
Then since $\beta_{s, t}$ is the coefficient of $z_s z_t$, the following equation holds.
\begin{align*}
    &\varphi((A_{i_1}x^{d-{i_1}}y^{i_1})(A_{i_2}x^{d-{i_2}}y^{i_2})(C_{i_3}x^{d-{i_3}}y^{i_3}))\varphi((B_{i_4}x^{d-{i_4}}y^{i_4})(B_{i_5}x^{d-{i_5}}y^{i_5})(C_{i_6}x^{d-{i_6}}y^{i_6})) \\
    &= A_{i_1} A_{i_2} B_{i_4} B_{i_5} C_{i_3} C_{i_6} z_{i_1+i_2+i_3} z_{i_4+i_5+i_6} = A_\alpha  A_\gamma B_\beta  B_\delta C_\lambda C_\mu z_s z_t = D_1 z_s z_t
\end{align*}
Thus, if we choose two elements from either $\mathcal{I}_A$ or $\mathcal{I_B}$ and one element from $\mathcal{I_C}$, the sum of these three indices will be $s$ or $t$,
while the remaining three indices will have sum $t$ or $s$, respectively.
In other words, to select three indices $i, j, k$ whose sum equals $s$, we need to choose $i, j$ from either $\mathcal{I}_A$ or $\mathcal{I_B}$, and $k$ from $\mathcal{I_C}$, i.e.,
\[
\left( i, j \in \mathcal{I}_A \mspace{13mu} \text{or} \mspace{13mu} i, j \in \mathcal{I}_B \right) \quad \text{and} \quad k \in \mathcal{I}_C .
\]
Therefore, the six indices of $D(\alpha, \beta, \gamma, \delta, \lambda, \mu)$ must satisfy one of the following $2 \times 2=4$ conditions.
\begin{align*}
    &\textbf{$1$.} \quad \lambda+\alpha+\gamma=s  , \quad \mu+\beta+\delta=t \\
    &\textbf{$2$.} \quad \mu+\alpha+\gamma=s  , \quad \lambda+\beta+\delta=t \\
    &\textbf{$3$.} \quad \lambda+\beta+\delta=s  , \quad \mu+\alpha+\gamma=t \\
    &\textbf{$4$.} \quad \mu+\beta+\delta=s  , \quad \lambda+\alpha+\gamma=t
\end{align*}
\noindent \underline{Case 2:} Suppose that $D_1$ is obtained from $\varphi(fgh)^2$.
Again, since $\beta_{s, t}$ is the coefficient of $z_s z_t$, the following equation holds.
\begin{align*}
    &\varphi((A_{i_1}x^{d-{i_1}}y^{i_1})(B_{i_2}x^{d-{i_2}}y^{i_2})(C_{i_3}x^{d-{i_3}}y^{i_3}))\varphi((A_{i_4}x^{d-{i_4}}y^{i_4})(B_{i_5}x^{d-{i_5}}y^{i_5})(C_{i_6}x^{d-{i_6}}y^{i_6})) \\
    &= A_{i_1} A_{i_4} B_{i_2} B_{i_5} C_{i_3} C_{i_6} z_{i_1+i_2+i_3} z_{i_4+i_5+i_6} = A_\alpha  A_\gamma B_\beta  B_\delta C_\lambda C_\mu z_s z_t = D_1 z_s z_t
\end{align*}
Thus, if we choose one element each from $\mathcal{I}_A$, $\mathcal{I_B}$ and $\mathcal{I_C}$, the sum of these three indices will be $s$ or $t$,
while the remaining three indices will have sum $t$ or $s$, respectively.
In other words, to select three indices $i, j, k$ whose sum equals $s$, we need to choose $i$ from $\mathcal{I}_A$, $j$ from $\mathcal{I_B}$, and $k$ from $\mathcal{I_C}$, i.e.,
\[
i \in \mathcal{I}_A \quad \text{and} \quad j \in \mathcal{I}_B  \quad \text{and} \quad k \in \mathcal{I}_C .
\]
Therefore, the six indices of $D(\alpha, \beta, \gamma, \delta, \lambda, \mu)$ must satisfy one of the following $2^3=8$ conditions.
\begin{align*}
    &\textbf{$1$.} \quad \lambda+\alpha+\beta=s  , \quad \mu+\gamma+\delta=t \\
    &\textbf{$2$.} \quad \lambda+\alpha+\delta=s  , \quad \mu+\gamma+\beta=t \\
    &\textbf{$3$.} \quad \lambda+\gamma+\beta=s  , \quad \mu+\alpha+\delta=t \\
    &\textbf{$4$.} \quad \lambda+\gamma+\delta=s  , \quad \mu+\alpha+\beta=t \\
    &\textbf{$5$.} \quad \mu+\alpha+\beta=s  , \quad \lambda+\gamma+\delta=t \\
    &\textbf{$6$.} \quad \mu+\alpha+\delta=s  , \quad \lambda+\gamma+\beta=t \\
    &\textbf{$7$.} \quad \mu+\gamma+\beta=s  , \quad \lambda+\alpha+\delta=t \\
    &\textbf{$8$.} \quad \mu+\gamma+\delta=s  , \quad \lambda+\alpha+\beta=t
\end{align*}
Combining the two cases, we conclude that the six indices of $D(\alpha, \beta, \gamma, \delta, \lambda, \mu)$ must satisfy at least one of the $12$ conditions listed above.
In fact, this is equivalent to choose two indices $i, j$ from $\mathcal{I}_A \cup \mathcal{I}_B$ and one index $k$ from $\mathcal{I}_C$ such that $i+j+k=s$, i.e.,
\[
i, j \in \mathcal{I}_A \cup \mathcal{I}_B  \quad \text{and} \quad k \in \mathcal{I}_C .
\]

The next step is to find another $D$ whose indices satisfy the condition \eqref{sumcon}.
In the previous step, we proved that, given $D(\alpha, \beta, \gamma, \delta, \lambda, \mu)$, there are three indices $i, j, k \in \mathcal{I}_A \cup \mathcal{I}_B \cup \mathcal{I}_C$ whose sum is $s$, and the remaining three indices have sum $t$.
Specifically, $i$ and $j$ are in $\mathcal{I}_A \cup \mathcal{I}_B$ and $k$ is in $\mathcal{I}_C$.
For the remaining three indices, let us denote the two from $\mathcal{I}_A \cup \mathcal{I}_B$ by $i'$ and $j'$, and the one from $\mathcal{I}_C$ by $k'$.
This gives us the following relation:
\[
i+j+k=s, \quad i'+j'+k'=t.
\]
And since $i, j, i', j' \in \mathcal{I}_A \cup \mathcal{I}_B$, they correspond to the first to fouth indices in $D(\alpha, \beta, \gamma, \delta, \lambda, \mu)$.
On the other hand, $k$ and $k'$ are the fifth and sixth indices in $D(\alpha, \beta, \gamma, \delta, \lambda, \mu)$.
First, we will fix $i, k, k'$ in the $1st, 5th, 6th$ positions.
Consider the following equalities, which follow immediately from $D(\alpha, \beta, \gamma, \delta, \lambda, \mu)=C_\lambda C_\mu q_{\alpha \beta} q_{\gamma \delta}$.
\begin{align*}
D(\alpha, \beta, \gamma, \delta, \lambda, \mu) &= - D(\beta, \alpha, \gamma, \delta, \lambda, \mu) \mspace{50mu}  ( \text{switching $\alpha$ and $\beta$} ) \\
D(\alpha, \beta, \gamma, \delta, \lambda, \mu) &= - D(\alpha, \beta, \delta, \gamma, \lambda, \mu) \mspace{50mu}  (\text{switching $\gamma$ and $\delta$}) \\
D(\alpha, \beta, \gamma, \delta, \lambda, \mu) &= D(\alpha, \beta, \gamma, \delta, \mu, \lambda)  \mspace{65mu} (\text{switching $\lambda$ and $\mu$}) \\
D(\alpha, \beta, \gamma, \delta, \lambda, \mu) &= D(\gamma, \delta, \alpha, \beta, \lambda, \mu) \mspace{60mu} (\text{switching $p_{\alpha \beta}$ and $p_{\gamma \delta}$})
\end{align*}
Using the four equalities above, we can easily fix $i, j, k$ in the desired positions, allowing us to express $D$ as $D(i, - , -, -, k, k' )$.
Next, we need to fix $j$ as the third index.
If $j$ is already the third or fourth index, we can use the same equalities as before.
If $j$ is the second index, we can apply the following Plücker relation to move $j$ to the third position.
\[
 D(\alpha, \gamma, \beta, \delta, \lambda, \mu) = -D(\alpha, \delta, \gamma, \beta, \lambda, \mu) + D(\alpha, \beta, \gamma, \delta, \lambda, \mu)
\]
If there is no relation because $i$ (or $j$) is equal to $i'$ or $j'$, we can also label $i'$ or $j'$ as $i$ (or $j$).
Now, we have $D=D(i, - , j, -, k, k' )$, and the proof is complete by reindexing the first through sixth indices by $\alpha, \beta, \gamma, \delta, \lambda, \mu$ respectivley.

\noindent $(2)$
By the definition of $\varLambda(s, t)$, it is clearly a subset of $\Sigma$ which is linearly independent.

\noindent $(3)$
    To show that the coefficient of any element $D(\alpha, \beta, \gamma, \delta, \lambda, \mu)$ in $\varLambda(s, t)$ is nonzero, let us consider the one of the four terms that appear when expanding $D(\alpha, \beta, \gamma, \delta, \lambda, \mu)$, denoted as $D_2$ below:
    \[
    D_2:=C_{\lambda} C_{\mu} A_\beta  A_\gamma B_\alpha  B_\delta.
    \]

    First, we claim that the coefficient of $D_2$ is nonzero if and only if $D(\alpha, \beta, \gamma, \delta, \lambda, \mu)$ has a nonzero coefficient.
    This will be proved once we show that any other $D:=D(-, -, -, -, -, -)$ cannot contain $D_2$ as one of its terms.
    Note that the six indices $\alpha, \beta, \gamma, \delta, \lambda, \mu$ appear in three cases:
    \begin{align} \label{threeD}
    D(\alpha, \beta, \gamma, \delta, \lambda, \mu), D(\alpha, \gamma, \beta, \delta, \lambda, \mu), D(\alpha, \delta, \gamma, \beta, \lambda, \mu) .
    \end{align}
    If $D$ does not contain all six indices, then it cannot contain $D_2$ as one of its terms.
    Therefore, we only need to consider these three cases.
    Moreover, if $\alpha, \beta, \gamma, \delta$ are not distinct, there is only one $D$ (the given $D$) left, so we are done.
    Thus, let us assume they are all distinct.
    It is clear that the third $D$ in \eqref{threeD} does not contain $D_2$ when expanding it.
    Since $D(\alpha, \beta, \gamma, \delta, \lambda, \mu)$ is an element of $\varSigma(s, t)$, we know that the indices satisfy the following condition:
    \begin{align*}
    \lambda+\alpha+\gamma=s ; \quad \mu+\beta+\delta=t .
    \end{align*}
    And since $D(\alpha, \beta, \gamma, \delta, \lambda, \mu)$ is also an element of $\varLambda(s, t)$, it holds that
    \[
    \alpha < \gamma < \beta, \delta \quad \text{and} \quad \lambda \leq \mu .
    \]
    This implies that the second $D$ in \eqref{threeD} is not in $\varSigma(s, t)$ even if we consider $4$ switchings in the proof of $(1)$.
    Therefore, only the first $D$ contains $D_2$ as one of its terms.

    Next, we will show that the coefficient of $D_2$ is nonzero.
    By the construction of $Q_\ell$-mapping, we know that the coefficient could be zero if $D_2$ is obtained from both $\varphi(f^2h)\varphi(g^2h)$ and $\varphi(fgh)^2$.
    It is clear from the proof of Case 2 of $(1)$ that $D_2$ can be constructed from $\varphi(fgh)^2$.
    However, it cannot be constructed from $\varphi(f^2h)\varphi(g^2h)$.
    Indeed, according to the proof of Case 1 of $(1)$, it means that three indices other than $\alpha, \gamma, \lambda$ have sum $s$, which contradicts the condition for $\varLambda(s, t)$ as above.
\end{proof}

\begin{corollary} \label{indalpha}
For a given $1 \leq \ell \leq d/2$, if $\alpha_{0, i}=0$ for all $1 \leq i \leq 2\ell-1$, then $\alpha_{0, \ell'} = 0$ for any $\ell' > 2\ell-1$.
\end{corollary}

\begin{proof}
By Lemma \ref{lem:basic coefficients}(2) and Proposition \ref{alphabetapro}(1), the coefficients $\alpha_{0, k}$ can be expressed as follows:
\begin{align} \label{0k}
\alpha_{0, k} = \beta_{0, k+1} = \sum_{\begin{subarray}{c} \beta+\delta+\mu=k+1 \\ 1 \leq \beta, \delta \leq \ell, \ 0 \leq \mu \leq d-2\ell \end{subarray}} d(\beta, \delta, \mu) \ C_0 C_\mu q_{0 \beta} q_{0 \delta} .  \mspace{30mu} (d(\beta, \delta, \mu) \in \mathbb{K})
\end{align}
In particular, all the coefficients $d(\beta, \delta, \mu)$ are nonzero by Proposition \ref{alphabetapro}(3).

If $C_0=0$, the proof is complete since every term of $\alpha_{0, k}$ contains $C_0$ as a factor, which makes $\alpha_{0, k}=0$.
Therefore, let us assume that $C_0 \neq 0$.
We will now prove that the hypothesis of the corollary implies that $q_{0i}=0$ for all $i=1, \dots, \ell$.
Then this will show that $\alpha_{0, k}=0$, as every terms of it contains $q_{0i}$ as a factor as shown above.

To do this, let us use the induction on the index $i$ of $q_{0i}$.
By the hypothesis of the corollary and \eqref{0k}, we have:
\[
\alpha_{0, 1} = d(1, 1, 0) \ C_0^2 q_{0 1}^2 = 0.
\]
Since $C_0 \neq 0$, it follows that $q_{01}=0$.
Now, suppose $q_{0i}=0$ for all $i=1, \dots, j \ (\leq \ell)$.
Again by the hypothesis of the corollary and \eqref{0k}, we have:
\[
\alpha_{0, 2j+1} = d(j+1, j+1, 0) \ C_0^2 q_{0 (j+1)}^2 + \sum_{\begin{subarray}{c} \beta+\delta+\mu=2j+2 \\ 1 \leq \beta, \delta \leq \ell, \ 1 \leq \mu \leq d-2\ell \end{subarray}} d(\beta, \delta, \mu) \ C_0 C_\mu q_{0 \beta} q_{0 \delta} = 0 .
\]
By the induction hypothesis, this becomes
\[
\alpha_{0, 2j+1} = d(j+1, j+1, 0) \ C_0^2 q_{0 (j+1)}^2 = 0 .
\]
Similary, since $C_0 \neq 0$, this gives us $q_{0 (j+1)}=0$.
\end{proof}

Now we will demonstrate the existence of $\varLambda$ in all coordinates.
In fact, all the statements above can be generalized for the Veronese varieties except for this existence.

\begin{proposition} \label{nondegterm}
The set $\varLambda(i, j+1)$ in Proposition \ref{alphabetapro} is nonempty for all $0 \leq i < j \leq d-1$.
\end{proposition}

\begin{proof}
We will construct a specific term $D:=D(\alpha, \beta, \gamma, \delta, \lambda, \mu)$ in $\varLambda(i, j+1)$ and verify that it indeed belongs to $\varLambda(i, j+1)$.
Specifically, the term should satisfy the following conditions:
\begin{align*}
    &(1) \quad 0 \leq \alpha \leq \gamma \leq \beta, \delta \leq \ell,  \ \alpha \neq \beta, \ \gamma \neq \delta, \ 0 \leq \lambda, \mu \leq d-2\ell; &(\because \ D \in \Sigma) \\
    &(2) \quad \lambda+\alpha+\gamma = i , \quad \mu+\beta+\delta=j+1.  &(\because \ D \in \varSigma(i, j+1))
\end{align*}
\\
\noindent \underline{Case 1:} Suppose $j+1 \leq 2l$. Let $\lambda=\mu=0$ and let
\begin{align*}
\begin{cases}
    \alpha=m, \beta=n, \gamma=m, \delta=n & \text{if $i$ is even and $j+1$ is even} \\
    \alpha=m, \beta=n, \gamma=m, \delta=n+1 & \text{if $i$ is even and $j+1$ is odd} \\
    \alpha=m, \beta=n, \gamma=m+1, \delta=n & \text{if $i$ is odd and $j+1$ is even} \\
    \alpha=m, \beta=n, \gamma=m+1, \delta=n+1 & \text{if $i$ is odd and $j+1$ is odd}
\end{cases}
\end{align*}
with $m=\lfloor \frac{i}{2} \rfloor$ and $n=\lfloor \frac{j+1}{2} \rfloor$.
We need to prove that all the indices of $D$ satisfy the two conditions $(1)$ and $(2)$ mentioned above.
Since  the condition $(2)$ is straightforward to verify, we omit its proof here.
Thus, it remains to show that the condition $(1)$ holds.
As $\lambda=\mu=0$, it is obvious that the two indices satisfy $(1)$.
For the remaining indices $\alpha, \beta, \gamma, \delta$, let us consider relation between $i$ and $j$.
Since $i<j$, we have $i+2 \leq j+1$, which implies that $m < m+1 \leq n$.
As $m$ is clearly nonnegative, the construction gives us that
\[
0 \leq \alpha \leq \gamma \leq \beta \leq \delta \quad \textup{and} \quad \alpha < \beta.
\]
Additionally, since $j+1 \leq 2\ell$ by the hypothesis, it follows that $\delta \leq \ell$.
The remaining task is to show that $\gamma \neq \delta$.
Since $m < n$, this can be proven easily, except in the case when $i$ and $j$ are both odd.
If $i$ and $j$ are both odd and $i<j$, we have $i+3 \leq j+1$, which gives us
\[
m+1 \ = \ \left\lfloor \frac{i}{2} \right\rfloor +1 \  = \ \left\lfloor \frac{i+2}{2} \right\rfloor \  < \ \left\lfloor \frac{i+3}{2} \right\rfloor \  \leq \ \left\lfloor \frac{j+1}{2} \right\rfloor \ = \ n.
\]
\\
\noindent \underline{Case 2:} Suppose $j+1 > 2l$.
First, set $\beta=\delta=\ell$ and $\mu=j+1-2\ell$, and assume that $(\alpha, \gamma, \delta)$ is a parition of $i$ such that
\[
0 \leq \alpha < \beta, \quad 0 \leq \gamma < \delta, \quad 0\leq \lambda \leq \mu.
\]
Then by the construction, all the indices satisfy the conditions $(1)$ and $(2)$.
Thus, it remains to show that such a partition actually exists.
To do this, it is enough to show that the maximum possible values of $\alpha, \gamma, \lambda$ add up to at least $j-1$.
This can be shown as follows:
\[
\alpha+\gamma+\lambda \ \leq \ (\beta-1)+(\delta-1)+\mu \ = \ (\ell-1)+(\ell-1)+(j+1-2\ell) \ = \ j-1.
\]
\end{proof}

\section{Proofs of Theorem \ref{thm:minimal and birational}} \label{3}
\noindent This section is devoted to giving a proof of Theorem \ref{thm:minimal and birational}.\\

\noindent {\bf Proof of Theorem \ref{thm:minimal and birational}.} $(1)$
Here we will prove that the irreducible decomposition of $\Phi_3 (\mathcal{C}_d )$ in Theorem \ref{thm:minimal and birational}.$(1)$ is minimal. Let us consider two distinct components $W_{\ell_1} (\mathcal{C}_d)$ and $W_{\ell_2} (\mathcal{C}_d)$ with $\ell_1 < \ell_2$. Since these two components are both irreducible and have the same dimension (cf. Definition and Remark \ref{definition of Q map}.$(3)$), minimality will be established once we show that $W_{\ell_2} \nsubseteq W_{\ell_1}$.
We will do this by finding a specific point in $W_{\ell_2} (\mathcal{C}_d)-W_{\ell_1} (\mathcal{C}_d)$.
For simplicity, let us focus only on the coefficients of
$$\{Q_{0,k}=z_0 z_{k+1}-z_1 z_{k} \mid 1 \leq k \leq 2\ell_2-1 \}.$$
We claim that there exists a point in $W_{\ell_2} (\mathcal{C}_d)-W_{\ell_1} (\mathcal{C}_d)$ where all coefficients are zero except for that of $Q_{0, 2\ell_2-1} = z_0 z_{2\ell_2}-z_1 z_{2\ell_2 -1}$.
That is,
\begin{align} \label{zerononzero}
\alpha_{0, 1}=\cdots=\alpha_{0, 2\ell_{2}-2}=0 \quad \textup{and} \quad \alpha_{0, 2\ell_2-1} \neq 0.
\end{align}
First we will explicitly construct such a point in $W_{\ell_2} (\mathcal{C}_d)$. Consider a quadric in $W_{\ell_2}$ given as follows.
\begin{align*}
Q_{\ell_2}(s^{\ell_2}, t^{\ell_2}, s^{d-2\ell_2}) &= z_{0}z_{2 \ell_2} - z_{\ell_2}^2 \\
                                                  &= (z_{0}z_{2\ell_2}-z_1 z_{2\ell_2-1}) + (z_{1}z_{2\ell_2-1}-z_2 z_{2\ell_2-2}) + \cdots + (z_{\ell_2-1}z_{\ell_2+1}-z_{\ell_2}^2) \\
                                                  &= Q_{0, 2\ell_2-1} + Q_{1, 2\ell_2-2} + \cdots + Q_{\ell_2-1, \ell_2}
\end{align*}
It is obvious that this quadric satisfies the condition \eqref{zerononzero}.
On the other hand, there is no point in $W_{\ell_1} (\mathcal{C}_d)$ that satisfies the condition \eqref{zerononzero}.
In fact, by Corollary \ref{indalpha}, if $\alpha_{0, i}$'s are all zero for $i=1$ to $2\ell_{2}-2$, then
\[
 \alpha_{0, 1}=\cdots=\alpha_{0, 2\ell_{1}-1}=0
\]
since $\ell_1 < \ell_2$. Therefore, all coordinates in $W_{\ell_1}$ are zero.

\noindent $(2)$ To show the nondegeneracy of $W_\ell(\mathcal{C}_d)$, it is enough to prove that all coordinates are $\mathbb{K}$-linearly independent. Since all coordinates have their terms in
\[
\Delta_\ell = \mathbb{K}[ \{ q_{\alpha \beta}  \mid 0 \leq \alpha < \beta \leq \ell \} ]_2 \otimes \mathbb{K}[C_0, \dots, C_{d-2\ell}]_2
\]
by Definition and Remark \ref{definition of Q map}(4), the Plücker relations should be considered. More precisely, there are linear relations among $D(\alpha, \beta, \gamma, \delta, \lambda, \mu)$'s:
\begin{align*}
&D(\alpha, \beta, \gamma, \delta, \lambda, \mu)-D(\alpha, \gamma, \beta, \delta, \lambda, \mu)+D(\alpha, \delta, \beta, \gamma, \lambda, \mu) = 0
\end{align*}
where $0 \leq \alpha < \beta < \gamma < \delta \leq \ell$ and $0 \leq \lambda, \mu \leq d-2\ell$. This corresponds to the Plücker relation
\[
p_{\alpha \beta} p_{\gamma \delta} - p_{\alpha \gamma} p_{\beta \delta} + p_{\alpha \delta} p_{\beta \gamma} = 0
\]
in the obvious way.
Note that there is a partition of $\alpha_{i, j}$'s which is decided by their index sum, namely
\[
M_{k}:= \{ \alpha_{i, j} \mid i+j=k \}
\]
for $1 \leq k \leq 2d-3$. For the same $k$, we further consider the set
\[
N_{k}:= \{ D(\alpha, \beta, \gamma, \delta, \lambda, \mu) \mid \alpha+\beta+\gamma+\delta+\lambda+\mu = k+1 \}
\]
which consists of all possible terms in $\alpha_{i, j} \in M_k$. It is obvious that all relations are made among elements in the same $N_k$. In other words, for any $k \neq k'$
\[
\langle N_k \rangle \cap \langle N_{k'} \rangle = 0.
\]
Hence we reduce the problem to show linear independence between all $\alpha_{i, j}$'s in $M_k$. To do this, fix the index sum $k \leq 2d-3$ and consider all coordinates in $M_k$:
\begin{align*}
\begin{cases}
\alpha_{0,k} \ , \ \alpha_{1, k-1} \ , \ \dots \ , \ \alpha_{\lceil \frac{k}{2}-1 \rceil, \lfloor \frac{k}{2}+1 \rfloor}            &\text{if $k \leq d-1$;} \\
\alpha_{k-d+1,d-1} \ , \ \alpha_{k-d+2, d-2} \ , \ \dots \ , \ \alpha_{\lceil \frac{k}{2}-1 \rceil, \lfloor \frac{k}{2}+1 \rfloor}  &\text{if $k \geq d$.}
\end{cases}
\end{align*}
Lemma \ref{lem:basic coefficients} and Proposition \ref{alphabetapro} tell us that the set $N_k$ can be divided into $\varSigma (i, k+1-i)$'s with $0 \leq i \leq \lceil \frac{k}{2}-1 \rceil$, and that each $\alpha_{i, k-i}$ is in the spanning set
\begin{align*}
\begin{cases}
\varSigma_{i} := \langle \varSigma (0, k+1) \cup \varSigma (1, k) \cup \cdots \cup \varSigma (i, k+1-i) \rangle            &\text{if $k \leq d-1$;} \\
\varSigma_{i} := \langle \varSigma (k-d+1, d) \cup \varSigma (k-d+2, d-1) \cup \cdots \cup \varSigma (i, k+1-i) \rangle  &\text{if $k \geq d$.}
\end{cases}
\end{align*}
Hence it suffices to show that for any $i$ the coordinate $\alpha_{i, j}$ has a term in $\varSigma (i, k+1-i = j+1)$ not in $\varSigma_{i-1}$ which induces a filtration of $\mathbb{K}$-vector spaces
\begin{align*}
\begin{cases}
\varSigma_{0} \ \subsetneq \ \varSigma_{1} \ \subsetneq \ \cdots \ \subsetneq \ \varSigma_{i-1} \ \subsetneq \ \varSigma_{i} \ \subsetneq \ \cdots            &\text{if $k \leq d-1$;} \\
\varSigma_{k-d+1} \ \subsetneq \ \varSigma_{k-d+2} \ \subsetneq \ \cdots \ \subsetneq \ \varSigma_{i-1} \ \subsetneq \ \varSigma_{i} \ \subsetneq \ \cdots &\text{if $k \geq d$.}
\end{cases}
\end{align*}
That is, $\alpha_{i, j} \in \varSigma_i - \varSigma_{i-1}$.

We claim that the element $D(\alpha, \beta, \gamma, \delta, \lambda, \mu)$ of $\varLambda(i, j+1)$, constructed in Proposition \ref{nondegterm}, is contained precisely in $\varSigma_i - \varSigma_{i-1}$.
First, since $D(\alpha, \beta, \gamma, \delta, \lambda, \mu)$ is an element of $\varLambda(i, j+1)$, it is also in $\varSigma_i$.
To ensure that $D(\alpha, \beta, \gamma, \delta, \lambda, \mu)$ is not included in $\varSigma_{i-1}$, it must not be a linear combination of elements contained in $\varSigma_{i-1}$.
Therefore, it is enough to show that none of three terms in the following Plücker relation (which contains $D(\alpha, \beta, \gamma, \delta, \lambda, \mu)$) are included in $\varSigma_{i-1}$.
\[
D(\alpha, \beta, \gamma, \delta, \lambda, \mu) - D(\alpha, \gamma, \beta, \delta, \lambda, \mu) + D(\alpha, \delta, \beta, \gamma, \lambda, \mu) = 0
\]
If any $D:=D(-, -, -, -, -, -)$ is included, by Proposition \ref{alphabetapro}(1), it means that two of $\alpha, \beta, \gamma, \delta$ and one of $\lambda, \mu$ have sum less than $i$.
This contradicts the conditions of $\varLambda(i, j+1)$:
\[
\alpha \leq \gamma \leq \beta, \delta, \quad \lambda \leq \mu, \quad \alpha+\gamma+\lambda=i.
\]

\noindent $(3)$ The morphism
$$\widetilde{Q_1} :  \mathbb{P}^{d-2} \rightarrow  \P^{{d \choose 2}-1}$$
is defined by the subset $\{ \alpha_{ij} ~|~ 0 \leq i < j \leq d-1 \}$ of $H^0 (  \P^{d-2} ,  \mathcal{O}_{\P^{d-2} } (2))$. Since the former subset is $\mathbb{K}$-linearly independent and the latter $\mathbb{K}$-vector space is of dimension ${d \choose 2}$, it follows that $\widetilde{Q_1}$ is the 2nd Veronese embedding of $\mathbb{P}^{d-2}$ and so $W_1 = \nu_2(\mathbb{P}^{d-2})$.   \qed \\

\section{Proof of Theorem \ref{thm:Singularity of W_ell (X))}}
\noindent In this section we will provide a proof of Theorem \ref{thm:Singularity of W_ell (X))}. We begin with fixing a few notations related to the $\widetilde{Q_{\ell}}$-maps of the Veronese varieties.

\begin{notation and remakrs}\label{not and rmk:Q-morphism}
Let $X = \nu_d (\P^n ) \subset \P^r ,\quad r = {{n+d} \choose {n}}-1$, be the $d$th Veronese embedding of $\P^n$.
Let $R = \K [x_0 ,  \ldots , x_n ]$ and $S = \K [ \{ z_I \mid I \in \mathcal{I}_d \} ]$ be respectively the homogeneous coordinate rings of $\P^n$ and $\P^r$,
where
\[
\mathcal{I}_d := \{ (a_0, \dots, a_n) \in \Z_{\geq 0}^{n+1} \mid a_0+\cdots+a_n=d \} .
\]
Thus there is a canonical isomorphism $\varphi : R_d \rightarrow S_1$ as $\mathbb{K}$-vector spaces. Now, let $e$ be the integer such that $d=2e$ or $d=2e+1$. For each $1 \leq \ell \leq e$, the morphism
$$\widetilde{Q_{\ell}} : \mathbb{G} (1,\P^p ) \times \P^q \rightarrow W_{\ell} (X)$$
is induced from the map
\begin{equation}\label{eq:Q-morphism}
Q_{\ell} : H^0 (\P^n , \mathcal{O}_{\P^n} (\ell)) \times H^0 (\P^n , \mathcal{O}_{\P^n} (\ell)) \times H^0 (\P^n , \mathcal{O}_{\P^n} (d-2\ell)) \rightarrow I(X)_2
\end{equation}
defined as
$$Q_{\ell} (f,g,h) = \varphi (f^2 h) \varphi (g^2 h ) - \varphi (fgh)^2 .$$
We already know the following facts about $W_{\ell} (X)$. The morphism $\widetilde{Q_{\ell}}$ in (\ref{eq:Q-morphism}) is defined by a base point free linear series contained in
$$V := H^0 (\mathbb{G} (1,\P^{p} ) \times \P^{q} ,  \mathcal{O}_{\mathbb{G} (1,\P^{p} ) \times \P^{q} } (2,2)).$$
More precisely, $W_{\ell} (X)$ in the projective space $\P (I(X)_2 )$ is the image of a linear projection of the linearly normal variety
$$\mathbb{G} (1,\P^{p} ) \times \P^{q} \subset \P (V)$$
since $\mathcal{O}_{\mathbb{G} (1,\P^{p} ) \times \P^{q} } (2,2)$ is a very ample line bundle of $\mathbb{G} (1,\P^{p} ) \times \P^{q}$.
\end{notation and remakrs}

To prove Theorem \ref{thm:Singularity of W_ell (X))}, we begin with the following lemma.

\begin{lemma} \label{injlem}
Let $q \in W_{\ell} (X)$ be a point. Then the following two conditions are equivalent:
\begin{enumerate}
\item[$(i)$] There exists a positive integer $s \leq min\{\ell, \lfloor \frac{d-2\ell}{2} \rfloor \}$ such that
$$q = \widetilde{Q_\ell} (\langle Cf' , Cg' \rangle , \langle D^2 h' \rangle)$$
for some $f', g' \in R_{\ell-s}$, $h' \in R_{d-2\ell-2s}$ and $\mathbb{K}$-linearly independent forms $C,D \in R_s$;
\item[$(ii)$] $\widetilde{Q_\ell} ^{-1} (q)$ has at least two distinct elements.
\end{enumerate}
\end{lemma}

\begin{proof}
$(i) \Longrightarrow (ii)$ : Note that
$$\widetilde{Q_\ell} (\langle Cf' , Cg' \rangle , \langle D^2 h' \rangle) = \widetilde{Q_\ell} (\langle Df' , Dg' \rangle , \langle C^2 h' \rangle)$$
Also, $(\langle Cf' , Cg' \rangle , \langle D^2 h' \rangle)$ and $(\langle Df' , Dg' \rangle , \langle C^2 h' \rangle)$ are different elements of $ \mathbb{G}(1, \mathbb{P}^{p}) \times \mathbb{P}^{q}$ since $C$ and $D$ are $\mathbb{K}$-linearly independent. Therefore $\widetilde{Q_\ell} ^{-1} (q)$ has more than two distinct elements.

\noindent $(ii) \Longrightarrow (i)$ : Before we begin, let us clarify a few things.

We will use the lexicographical order for monomials of $R$, with $x_0 > x_1 > \dots > x_n$. Additionally, we will apply the lexicographical monomial order in $S$, where $z_I > z_J$ if $I > J$ in the lexicographical order on $\Z^{n+1}_{\geq 0}$.
Note that $a > b$ in $R_d$ if and only if $\varphi(a) > \varphi(b)$ in $S_1$ where $\varphi: R_d \rightarrow S_1$ is the isomorphism in Notation and Remarks \ref{not and rmk:Q-morphism}.

For a point $(\langle f, g \rangle, \langle h \rangle)$ of $\mathbb{G} (1,\P^{p} ) \times \P^{q}$ where $f, g \in R_\ell$ and $h \in R_{d-2\ell}$, we always assume that
the leading coefficients of $f, g$ and $h$ are $1$, and that $f$ and $g$ have degrees satisfying
\begin{equation}\label{eq:multidegree}
{\rm multideg}(f) > {\rm multideg}(g).
\end{equation}
This is possible since $\langle f, g \rangle$ is an element of $\mathbb{G} (1,\P^{p} )$.

Now, suppose that $\widetilde{Q_\ell} ^{-1} (q)$ contains two distinct points $(\langle f, g \rangle, \langle h \rangle)$ and $(\langle u, v \rangle, \langle w \rangle)$. Then it holds that
\begin{align} \label{eqqqq}
\varphi(u^2 w)\varphi(v^2 w) - \varphi(uvw)^2 = k \varphi(f^2 h)\varphi(g^2 h) - k \varphi(fgh)^2
\end{align}
for some $k \in \mathbb{K}$.
Also, the assumption (\ref{eq:multidegree}) implies that
$${\rm multideg}(f^2h) \ > \ {\rm multideg}(fgh) \ > \ {\rm multideg}(g^2h)$$
and
$${\rm multideg}(u^2w) \ > \ {\rm multideg}(uvw) \ > \ {\rm multideg}(v^2w).$$
Note that we can write $\varphi(u^2w)$ and $\varphi(f^2h)$ as
\[
\varphi(u^2w) = z_{{\rm multideg}(u^2w)}+\mbox{(lower terms)} \quad \mbox{and} \quad \varphi(f^2h) = z_{{\rm multideg}(f^2h)}+\mbox{(lower terms)} .
\]
By comparing the multidegrees in the equality \eqref{eqqqq}, we can simplify it as follows.
\begin{align} \label{redeq}
z_{{\rm multideg}(u^2w)} \varphi(v^2w) - k z_{{\rm multideg}(f^2h)} \varphi(g^2h) = 0
\end{align}
Thus, to satisfy \eqref{redeq}, the multidegrees of $f^2h$ and $u^2w$ must be same.
Since the leading coefficients of $\varphi(v^2w)$ and $\varphi(g^2h)$ are both $1$, this implies that $k=1$ and $v^2w=g^2h$.
Substituting $g^2h$ with $v^2w$ in \eqref{eqqqq}, we get
\begin{align} \label{changedeq}
\varphi(u^2w - f^2h)\varphi(v^2w) = \varphi(uvw+fgh)\varphi(uvw-fgh)  .
\end{align}
Let $g=g' \cdot C$ and $v=g' \cdot D$ where $g'=gcd(g, v)$.
We will prove that $C$ (and hence $D$) is nonconstant, $D^2$ divides $h$, and $C$ divides $f$.
First, suppose that $C$ is a constant.
This implies that $v=g$ and $w=h$, since $v^2w=g^2h$ and all the leading coefficients of $v, w, g, h$ are $1$.
Thus, we can rewrite \eqref{changedeq} as
\[
\varphi(u^2 w - f^2 w)\varphi(v^2 w) = \varphi(uvw+fvw)\varphi(uvw-fvw) ,
\]
which implies that
\[
u^2 w - f^2 w = k (uvw + fvw) \quad \text{or} \quad u^2 w - f^2 w = k(uvw - fvw)
\]
for some $k \in \mathbb{K}$. Factoring this, we get
\[
w(u+f)(u-f) = kvw (u+f) \quad \text{or} \quad w(u+f)(u-f) = kvw(u-f),
\]
and after cancelling common factors on both sides, we obtain
\[
f = u-kv \quad \text{or} \quad f = -u+kv .
\]
In consequence, we have $\langle u, v \rangle = \langle f, g \rangle$ and $\langle w \rangle = \langle h \rangle $.
Obviously, this means that the two points $(\langle f, g \rangle, \langle h \rangle)$ and $(\langle u, v \rangle, \langle w \rangle)$ are not distinct, which is a contradition.
Therefore, $C$ and $D$ are not constants.
Next, the equality $v^2w=g^2h$ gives us $D^2w=C^2h$.
This shows that $h$ is divisible by $D^2$ since $C$ and $D$ are relatively prime.
Now, we will show that $f$ is divisible by $C$.
To this aim, we will divide the proof into two cases: when both sides of equation \eqref{changedeq} are zero and when they are nonzero.
If both sides of \eqref{changedeq} are nonzero, then we have
\[
v^2w = k(uvw+fgh) \quad  \text{or} \quad v^2w = k(uvw-fgh)
\]
for some $k \in \mathbb{K}$.
Multiplying by $v$ both sides, we get
\[
v^3w = kuv^2w + kfghv  \quad  \text{or} \quad   v^3w = kuv^2w - kfghv .
\]
Applying the substitution $v^2w=g^2h$, this becomes
\[
g^2h(ku-v) = - kfghv \quad  \text{or} \quad g^2h(ku-v) = + kfghv .
\]
By cancelling common factors on both sides, it follows that
\[
g(ku-v) = -kfv \quad \text{or} \quad g(ku-v) = kfv .
\]
Since $C$ and $v$ are relatively prime, we conclude that $C$ divides $f$.
On the other hand, if both sides of \eqref{changedeq} are zero, then $uvw = fgh$ or $uvw = -fgh$.
As above, by multiplying by $v$ both sides and applying the substitution $v^2w=g^2h$, we can show that $C$ divides $f$.

Combining the facts proven above, $(\langle f, g \rangle, \langle h \rangle)$ can be rewritten as
\[
(\langle f, g \rangle, \langle h \rangle) = (\langle C  f', C  g' \rangle, \langle D^2  h' \rangle),
\]
which completes the proof.
\end{proof}

\noindent \textbf{Proof of Theorem \ref{thm:Singularity of W_ell (X))}} $(1)$ Suppose the morphism $\widetilde{Q_1}$ is not injective at some point $(\langle f, g \rangle, \langle h \rangle) \in \G \times \P$, then by Lemma \ref{injlem}, $f$ and $g$ should have a common factor. This is impossible since $f$ and $g$ are $\mathbb{K}$-linearly independent linear forms. In the same manner, suppose the morphism $\widetilde{Q_e}$ is not injective at some point $(\langle f, g \rangle, \langle h \rangle) \in \G \times \P$, then by Lemma \ref{injlem}, $h$ has a multiple factor. But for $h \in R_1$, this is a contradiction.

\noindent $(2)$ Since $\widetilde{Q_\ell}$ is a birational morphism, the image $W_\ell (X)$ will be a singular variety whenever the morphism is not injective at some point. Lemma \ref{injlem} shows that the singular locus of $W_\ell (X)$ contains the image of the subset
$$\Sigma_m := \{ (\langle Cf', Cg' \rangle, \langle M^2h' \rangle) \mid f', g' \in R_{\ell-m}, \ h' \in R_{d-2\ell-2m} \ \text{and} \ C, M \in R_m \}$$
of $\G(1, \P(R_\ell)) \times \P (R_{d-2\ell} )$ for all $1 \leq m \leq min\{\ell, \lfloor \frac{d-2\ell}{2} \rfloor \}$. Also, $\Sigma_m$ is the image of the following finite morphism
\begin{align*}
\P(R_m) \times \G(1, \P(R_{\ell-m})) \times \P(R_m) \times \P(R_{d-2\ell-2m}) \mspace{5mu} &\longrightarrow \mspace{5mu} \G(1, \P(R_\ell)) \times \P (R_{d-2\ell} ).\\
([C], \langle f', g' \rangle ,[M], [h'] ) \mspace{50mu} &\longmapsto \mspace{10mu} (\langle Cf', Cg' \rangle , [ M^2 h'])
\end{align*}
Thus the dimension of $Sing(W_\ell) (X)$ is at least
$$\dim \P(R_m) + \dim \G(1, \P(R_{\ell-m}))+\dim \P(R_m) + \dim \P(R_{d-2\ell-2m}),$$
which is equal to
$$ \binom{n+m}{n} -1 + 2\binom{n+\ell-m}{n}-4 + \binom{n+m}{n} -1 + \binom{n+d-2\ell-2m}{n}-1 .$$
In consequence, it holds that
\begin{align*}
\dim {\rm Sing}(W_\ell (X)) \geq \max_{m} \left\{2\binom{n+m}{n}+2\binom{n+\ell-m}{n}+\binom{n+d-2\ell-2m}{n}-7 \right\}.
\end{align*}
Finally, the desired inequality about $\dim {\rm Sing}(W_\ell (X))$ for the case $n=1$ follows directly from the above result. \qed \\

\section{Examples and further discussions}
\noindent Let $X = \nu_d (\P^n ) \subset \P^r$ be the $d$th Veronese embedding of $\P^n$. In Theorem \ref{thm:Singularity of W_ell (X))}.$(1)$, it is shown that the morphism
$$\widetilde{Q_e} : \mathbb{G} (1,\P^p ) \times \P^q \rightarrow W_e (X)$$
is always injective.
So, the next natural question is whether $\widetilde{Q_{\ell}}$ is an isomorphism or not.
In this section, we will solve this question through calculations for the cases where $(n,d)$ is equal to $(1,4)$ and $(1,5)$.
Indeed, $\widetilde{Q_2}$ is an isomorphism if $(n,d)=(1,4)$ (see Example \ref{d even}.$(1)$) while it is not if $(n,d) = (1,5)$ (see Example \ref{ex:d=5,l=2}). In general cases, or even when $n=1$ and $d>5$, we do not know whether $W_e (X)$ is smooth or singular.

\begin{example} \label{d even}
Suppose that $d=2e$ is an even integer and $e \geq 2$. Then the map $\widetilde{Q_e}$ in Theorem \ref{thm:minimal and birational} becomes
$$\widetilde{Q_{e}} : \mathbb{G} (1,\P^e )    \rightarrow W_e (\mathcal{C}) \subset \mathbb{P}^{\binom{2e}{2}-1}.$$
This morphism is defined by the $\mathbb{K}$-linearly independent subset $\{ \alpha_{ij} ~|~ 0 \leq i < j \leq d-1 \}$ of $H^0 (\mathbb{G} (1,\P^e ) , \mathcal{O}_{\mathbb{G} (1,\P^e )}(2))$.
Note that
$$h^0(\mathbb{G} (1,\P^e ), \mathcal{O}_{\mathbb{G} (1,\P^e )}(2)) = \frac{1}{12} e (e+1)^2 (e+2).$$

\noindent $(1)$ When $e =2$, we have
$$h^0(\mathbb{G} (1,\P^2 ), \mathcal{O}_{\mathbb{G} (1,\P^2 )}(2)) = 6 = {4 \choose 2}$$
and hence $W_2 (\mathcal{C})$ is equal to the Veronese surface $\nu_2 (\mathbb{G} (1,\P^2 ))$.

\noindent $(2)$ If $e \geq 3$, then $h^0(\mathbb{G} (1,\P^e ), \mathcal{O}_{\mathbb{G} (1,\P^e )}(2))$ is strictly bigger than ${2e \choose 2}$ and hence the map
$$\widetilde{Q_{e}} : \mathbb{G} (1,\P^e ) \rightarrow \mathbb{P}^{\binom{2e}{2}-1}$$
is equal to a linear projection of the linearly normal embedding of $\mathbb{G} (1,\P^e )$ by the line bundle $\mathcal{O}_{\mathbb{G} (1,\P^e )}(2)$.
\end{example}

\begin{example}\label{ex:d=4}
Let $\mathcal{C} \subset \P^4$ be the standard rational normal curve of degree $4$. Then its homogeneous ideal $I(\mathcal{C})$ is minimally generated by the following six quadratic polynomials:
\begin{align*}
Q_{0, 1}=z_0z_2-{z_1}^2, \quad Q_{0, 2}=z_0z_3-{z_1}z_2, \quad Q_{0, 3}=z_0z_4-{z_1}z_3 ,\\
Q_{1, 2}=z_1z_3-{z_2}^2, \quad Q_{1, 3}=z_1z_4-{z_2}z_3, \quad Q_{2, 3}=z_2z_4-{z_3}^2 .
\end{align*}
Theorem \ref{thm:minimal and birational}.$(1)$ shows that $\Phi_3 (\mathcal{C})$ is the union of $W_1 (\mathcal{C})$ and $W_2 (\mathcal{C})$. Furthermore, $W_1 (\mathcal{C}) \subset \P^5$ and $W_2 (\mathcal{C}) \subset \P^5$ are Veronese surfaces by Theorem \ref{thm:minimal and birational}.$(3)$ and Example \ref{d even}.$(1)$. In this example, we will try to describe these two irreducible components of $\Phi_3 (\mathcal{C})$ very explicitly.

\noindent $(1)$ From $\mathcal{O}_{\P^1} (4) = \mathcal{O}_{\P^1} (1)^2 \otimes \mathcal{O}_{\P^1} (2)$, we need to express
$$Q_1(x, y, a x^2+bxy+cy^2)$$
as a linear combination of the above six quadratic polynomials. Indeed, one can check that
\begin{align*}
Q_1(x, y, a x^2+bxy+cy^2) & = (az_0 +bz_1 +cz_2) (az_2 +bz_3 +cz_4) - (az_1 +bz_2 +cz_3)^2  \\
&= {a}^2 Q_{0, 1} + {ab} Q_{0, 2} + {ac} Q_{0, 3} + (b^2-ac) Q_{1, 2} + {bc} Q_{1, 3} + {c^2} Q_{2, 3} .
\end{align*}
Thus $W_1 (\mathcal{C})$ can be parameterized as
\begin{equation}\label{eq:W1 para}
W_1(\mathcal{C}) =\left\{ [a^2 : ab : ac : b^2-ac : bc : c^2] \mid [a:b:c] \in \mathbb{P}^2 \right\} .
\end{equation}

\noindent $(2)$ From $\mathcal{O}_{\P^1} (4) = \mathcal{O}_{\P^1} (2)^2$, we need to express
$$Q_2 (A_0x^2+A_1xy+A_2y^2, B_0x^2+B_1xy+B_2y^2, 1)$$
as a linear combination of the above six quadratic polynomials. Indeed, one can check that
$$Q_2 (A_0x^2+A_1xy+A_2y^2, B_0x^2+B_1xy+B_2y^2, 1) = FG-H^2$$
where
$$F = A _{0}^{2} z _{0} +2A _{0} A _{1} z _{1} +(A _{1}^{2} +2A _{0} A _{2} )z _{2} +2A _{1} A _{2} z _{3} +A _{2}^{2} z _{4} , $$
$$G = B _{0}^{2} z _{0} +2B _{0} B _{1} z _{1} +(B _{1}^{2} +2B _{0} B _{2} )z _{2} +2B _{1} B _{2} z _{3} +B _{2}^{2} z _{4} $$
and
$$H = A_0 B_0 z_0 + (A_0 B_1 +A_1 B_0 )z_1 + (A_0 B_2 +A_1 B_1 + A_2 B_0 )z_2 +(A_1 B_2 +A_2 B_1 ) z_3 +A_2 B_2 z_4 . $$
Then it can be shown that $Q_2 (A_0x^2+A_1xy+A_2y^2, B_0x^2+B_1xy+B_2y^2, 1)$ is equal to
$${a}^2 Q_{0, 1} + {2ab} Q_{0, 2} + {b^2} Q_{0, 3} + (b^2+2ac) Q_{1, 2} + {2bc} Q_{1, 3} + {c^2} Q_{2, 3}$$
where $a= A_0B_1-A_1B_0$, $b = A_0B_2-A_2B_0$ and $c = A_1B_2-A_2B_)$. Thus $W_2 (\mathcal{C})$ can be parameterized as
\begin{equation}\label{eq:W2 para}
W_2(\mathcal{C}) =\{ [a^2 : 2ab : b^2 : b^2+2ac : 2bc : c^2] \mid [a:b:c] \in \mathbb{P}^2 \} .
\end{equation}

\noindent $(3)$ Let $D$ denote the intersection of $W_1 (\mathcal{C})$ and $W_2 (\mathcal{C})$. From (\ref{eq:W1 para}) and (\ref{eq:W2 para}), we get the homogeneous ideals $I ( W_1 (\mathcal{C}))$ and $I (W_2 (\mathcal{C}))$ of $W_1 (\mathcal{C})$ and $W_2 (\mathcal{C})$ in the polynomial ring $\mathbb{K} [z_0 , z_1 , z_2 , z_3 , z_4 , z_5 ]$. Let $I(D)$ denote the homogeneous ideal of $D$ in $\P^5$. Thus
$$I(D) = \sqrt{I ( W_1 (\mathcal{C}))+I (W_2 (\mathcal{C}))}.$$
By the computation using MACAULAY2 (cf. \cite{GS}), we can show that $I(D)$ contains the linear form $3z_2 - z_3$ and that $D$ in $V(3z_2-z_3) = \mathbb{P}^4$ is a rational normal quartic curve.

\noindent $(4)$ From $(1) \sim (3)$, we conclude that the minimal irreducible decomposition of $\Phi_3 (\mathcal{C})$ is of the form
$$\Phi_3(\mathcal{C}) = W_1 (\mathcal{C}) \cup W_2 (\mathcal{C})$$
where $W_1 (\mathcal{C})$ and $W_2 (\mathcal{C})$ are Veronese surfaces in $\P^5$ such that $W_1 (\mathcal{C}) \cap W_2 (\mathcal{C})$ is a rational normal quartic curve.
\end{example}

\begin{example}\label{ex:d=5,l=2}
Let $\mathcal{C} \subset \P^5$ be the standard rational normal curve of degree $5$. Thus its homogeneous ideal $I(\mathcal{C})$ is generated by the following ten quadratic polynomials.
\begin{align*}
&Q_{0, 1}=z_0z_2-{z_1}^2, \quad Q_{0, 2}=z_0z_3-{z_1}z_2, \quad Q_{0, 3}=z_0z_4-{z_1}z_3, \quad Q_{0, 4}=z_0z_5-{z_1}z_4 \\
&Q_{1, 2}=z_1z_3-{z_2}^2, \quad Q_{1, 3}=z_1z_4-{z_2}z_3, \quad Q_{1, 4}=z_1z_5-{z_2}z_4, \quad Q_{2, 3}=z_2z_4-{z_3}^2 \\
&Q_{2, 4}=z_2z_5-{z_3}z_4, \quad Q_{3, 4}=z_3z_5-{z_4}^2
\end{align*}
From $\mathcal{O}_{\P^1} (5) = \mathcal{O}_{\P^1} (2)^2 \otimes \mathcal{O}_{\P^1} (1)$, we need to express $Q_2 (u,v,w)$ for
$$u = A_0 x^2 + A_1 xy + A_2 y^2 , \quad v = B_0 x^2 + B_1 xy + B_2 y^2 \quad \mbox{and} \quad w = C_0 x + C_1 y$$
as a linear combination of the above ten quadratic polynomials. Write
$$Q_2(u, v, w) = \sum_{0 \le i < j \le 4} \alpha_{i,j} Q_{i,j}.$$
By expanding
\begin{align*}
Q_2(u, v, w)  &= \varphi\left((A_0 x^2 + A_1 xy + A_2 y^2)^2 (C_0 x + C_1 y )\right) \varphi\left((B_0 x^2 + B_1 xy + B_2 y^2)^2 (C_0 x + C_1 y )\right) \\
              &\quad - \varphi\left((A_0 x^2 + A_1 xy + A_2 y^2)(B_0 x^2 + B_1 xy + B_2 y^2)(C_0 x + C_1 y )\right)^2 ,
\end{align*}
one can check the following equalities:
    \begin{align*}
    \alpha_{0, 1} &= \beta_{0, 2} = q_{01}^2 C_0^2 ;  \\
    \alpha_{0, 2} &= \beta_{0, 3} = 2q_{01}q_{02}C_0^2+q_{01}^2C_0C_1 \\
    \alpha_{0, 3} &= \beta_{0, 4} = q_{02}^2C_0^2+2q_{01}q_{02}C_0C_1 \\
    \alpha_{0, 4} &= \beta_{0, 5} = q_{02}^2C_0C_1 \\
    \alpha_{1,2}  &= \beta_{0, 4} + \beta_{1, 3} = q_{02}^2C_0^2 +2q_{01}q_{12}C_0^2+2q_{01}q_{02}C_0C_1+q_{01}^2C_1^2 \\
    \alpha_{1, 3} &= \beta_{0, 5} + \beta_{1, 4} = 2q_{02}q_{12}C_0^2+2q_{02}^2C_0C_1 +2q_{01}q_{12}C_0C_1 +2q_{01}q_{02}C_1^2 \\
    \alpha_{1, 4} &= \beta_{1, 5} = 2q_{02}q_{12}C_0C_1+q_{02}^2C_1^2 \\
    \alpha_{2, 3} &= \beta_{1, 5} + \beta_{2, 4} = q_{12}^2C_0^2 +2q_{02}q_{12}C_0C_1+q_{02}^2C_1^2+2q_{01}q_{12}C_1^2 \\
    \alpha_{2, 4} &= \beta_{2, 5} = q_{12}^2C_0C_1+q_{02}q_{12}C_1^2 \\
    \alpha_{3, 4} &= \beta_{3, 5} = q_{12}^2C_1^2
    \end{align*}
Thus $W_2(\mathcal{C})$ is the image of the morphism from $\mathbb{P}^2 \times \mathbb{P}^1$ to $\mathbb{P}^9$ defined by
\begin{equation}\label{eq:parameterization Y}
([q_{0,1} : q_{0,2} : q_{1,2} ] , [C_0 : C_1 ] ) \mapsto [\alpha_{0, 1}:\alpha_{0, 2}:\alpha_{0, 3}:\alpha_{0, 4}:\alpha_{1, 2}:\alpha_{1, 3}:\alpha_{1, 4}:\alpha_{2, 3}:\alpha_{2, 4}:\alpha_{3, 4}].
\end{equation}
This shows that $Y:=W_2(\mathcal{C})$ is a linear projection $\pi$ of $X:=\nu_2(\mathbb{P}^2 \times \mathbb{P}^1) \subset \mathbb{P}^{17}$ to $\mathbb{P}^9$. Furthermore, $\pi : X \rightarrow Y$ is injective by Theorem \ref{thm:Singularity of W_ell (X))}.$(1)$. Consider the short exact sequence
$$0 \ \longrightarrow \ \mathcal{O}_Y \ \longrightarrow \ \pi_{*}\mathcal{O}_X \ \longrightarrow  \ \mathcal{F} := \pi_{*}\mathcal{O}_X / \mathcal{O}_Y \ \longrightarrow \ 0$$
induced by $\pi$. The hilbert polynomials of $X$ and $Y$ are respectively $4t^3+8t^2+5t+1$ and $4t^3+8t^2+5t-5$. We obtain the hilbert polynomial of $Y$ through MACAULAY2 computation using the parametrization of $Y$ in (\ref{eq:parameterization Y}) (cf. \cite{GS}). It follows that the hilbert polynomial of $\mathcal{F}$ is equal to
$$HP_{X}(t) - HP_{Y}(t) = (4t^3+8t^2+5t+1) - (4t^3+8t^2+5t-5) = 6 .$$
In consequence, $Y$ is singular and the number of singular points of $Y$ is at most six.
\end{example}


\begin{thebibliography}{0000000}
\bibitem[BCR]{BCR} W. Bruns, A. Conca, and T. R\"{o}mer, {\em Koszul homology and syzygies of Veronese subalgebras},
Math. Ann. (2011) 351, 761--779.

\bibitem[E]{E} D. Eisenbud, {\em The Geometry of Syzygies}, no.229, Springer-Velag New York, (2005)

\bibitem[EKS]{EKS} D. Eisenbud, J. Koh, and M. Stillman, {\em Determinantal equations for curves of high degree}, Amer. J. Math. 110 (1988), 513--539.

\bibitem[GS]{GS} D. R. Grayson and M. E. Stillman, \newblock{Macaulay 2} -- {A Software System for Research in Algebraic Geometry}. \newblock{http://www.math.uiuc.edu/Macaulay2/}

\bibitem[Gr1]{Gr1} M. Green,{\em Koszul cohomology and the geometry of projective varieties I}, J. Differ. Geom., 19 (1984), 125--171.

\bibitem[Gr2]{Gr2} M. Green, {\em Koszul cohomology and the geometry of projective varieties II}, J. Differ. Geom., 20 (1984), 279--289.

\bibitem[HLMP]{HLMP} K. Han, W. Lee, H. Moon and E. Park, {\em Rank $3$ quadratic generators of Veronese embeddings}, Compositio Math. 157 (2021), 2001--2025.

\bibitem[JPW]{JPW} T. Jozefiak, P. Pragacz and J. Weyman, {\em Resolutions of determinantal varieties and tensor complexes associated with
symmtrric and antisymmetric matrices}, Asterisque 87-88, (1981), 109--189.

\bibitem[OP]{OP} G. Ottaviani and R. Paoletti, {\em Syzygies of Veronese embeddings}, Compositio Math., 125 (2001), 31--37.

\bibitem[Park]{Park} E. Park, {\em On rank $3$ quadratic equations of projective varieties}, Trans. Amer. Math. Soc. 377 (2024), no. 3, 2049–-2064.

\bibitem[Pu]{Pu} M. Pucci, \emph{The Veronese variety and Catalecticant matrices}, J. Algebra 202 (1998), 72--95.

\bibitem[SS]{SS} J. Sidman and G. Smith, {\em Linear determinantal equations for all projective schemes}, Algebra and Number Theory, Vol. 5 (2011), no. 8, 1041--1061.

\bibitem[Su]{Su} S. Sullivant, \emph{Combinatorial symbolic powers}, J. Algebra 319 (2008), no. 1, 115–-142.

\bibitem[Mu]{Mu} S, Mukai, {\em  An introduction to invariants and moduli}, Cambridge University Press, Vol. 81. {2003}, 244--245.
\end{thebibliography}
\end{document}